\documentclass[a4paper,11pt]{amsart}
\usepackage{graphicx}
\usepackage[english]{babel}
\usepackage{amsmath}
\usepackage{amsfonts}
\usepackage{amssymb}
\usepackage{amsthm}
\usepackage{mathtools}
\usepackage{tikz}
\usepackage{tikz-cd}
\usepackage{tikz-3dplot}
\usepackage{caption}
\usepackage{float}
\usepackage{enumerate}
\usepackage{bbm}
\usepackage{subcaption}

\usepackage{hyperref}
\usepackage[capitalize]{cleveref}
\usepackage{thm-restate}

\newtheorem{thm}{Theorem}[section]
\newtheorem{letterthm}{Theorem}

\newtheorem{lem}[thm]{Lemma}
\newtheorem*{lem*}{Lemma}
\newtheorem{prp}[thm]{Proposition}
\newtheorem{letterprp}[letterthm]{Proposition}

\newtheorem{cor}[thm]{Corollary}

\theoremstyle{definition}

\newtheorem{asn}[thm]{Assumption}
\newtheorem{ex}[thm]{Example}
\newtheorem{rmk}{Remark}

\newtheorem{lettercon}[letterthm]{Construction}

\Crefname{lettercon}{Construction}{Constructions}
\Crefname{constructionx}{Construction}{Constructions}

\Crefname{prp}{Proposition}{Propositions}
\Crefname{cor}{Corollary}{Corollaries}
\Crefname{lettercor}{Corollary}{Corollaries}
\Crefname{letterprp}{Proposition}{Propositions}
\Crefname{lem}{Lemma}{Lemmas}
\Crefname{thm}{Theorem}{Theorems}
\Crefname{letterthm}{Theorem}{Theorems}
\Crefname{clm}{Claim}{Claims}
\Crefname{dfn}{Definition}{Definitions}
\Crefname{asn}{Assumption}{Assumptions}
\Crefname{rmk}{Remark}{Remarks}
\Crefname{ex}{Example}{Examples}
\Crefname{constr}{Construction}{Constructions}
\Crefname{assumptions}{Assumptions}{Assumptions}
\Crefname{corollaryx}{Corollary}{Corollaries}
\Crefname{theoremx}{Theorem}{Theorems}
\Crefname{prpx}{Proposition}{Propositions}

\newcommand{\cA}{\mathcal{A}}
\newcommand{\cB}{\mathcal{B}}
\newcommand{\cC}{\mathcal{C}}
\newcommand{\cD}{\mathcal{D}}
\newcommand{\cE}{\mathcal{E}}
\newcommand{\cJ}{\mathcal{J}}
\newcommand{\cI}{\mathcal{I}}
\newcommand{\cL}{\mathcal{L}}

\newcommand{\cM}[1]{\mathcal{S}^{#1}}
\newcommand{\sinf}{\mathcal{S}_\infty}
\newcommand{\cS}{\mathcal{S}}
\newcommand{\cP}{\mathcal{P}}
\newcommand{\cQ}{\mathcal{Q}}
\newcommand{\cU}{\mathcal{U}}
\newcommand{\bbR}{\mathbb{R}}
\newcommand{\bbZ}{\mathbb{Z}}

\newcommand{\K}{\mathbb{K}}  % Field
\newcommand{\mult}[2]{\mathrm{mult}(#1, #2)}
\newcommand{\intMod}[1]{\K_{#1}}  % Interval module
\newcommand{\rk}{\mathrm{rk}}  % Generalized rank

\newcommand{\lbracket}{\left[}
\newcommand{\rbracket}{\right]}
\newcommand{\amspand}{&}

\newcommand{\comment}[1]{}

\DeclareMathOperator{\im}{im}
\DeclareMathOperator{\colim}{colim}
\DeclareMathOperator{\coker}{coker}
\DeclareMathOperator{\Hom}{Hom}
\DeclareMathOperator{\Obj}{Obj}

\newcommand{\Vect}{\mathrm{Vect}}
\newcommand{\vect}{\mathrm{vect}}
\renewcommand{\mod}{\mathrm{mod}\,}
\newcommand{\Mod}{\mathrm{Mod}\,}

\newcommand{\setcat}{\mathrm{Set}}
\newcommand{\funct}[2]{\mathrm{Fun(#1,#2)}}

\newcommand{\id}{\mathrm{id}}  % For identity

\newcommand{\up}{{\uparrow}}  % For upset
\newcommand{\down}{{\downarrow}}
\newcommand{\defeq}{\coloneqq}  % For definitions (:=)

\def\noteson{%
\gdef\samuel##1{\noindent{\color{blue}[Samuel: ##1]}}%
\gdef\justin##1{\noindent{\color{red}[Edit: ##1]}}%
\gdef\thomas##1{\noindent{\color{olive}[Thomas: ##1]}}%
\gdef\todo##1{\noindent{\color{violet}[to do: ##1]}}}

\noteson

\title{Generalized Rank via Minimal Subposet}
 % MSC 2020 codes

\author{Thomas Br\"ustle}
\author{Justin Desrochers}
\author{Samuel Leblanc}

\AtEndDocument{\bigskip{\small%
\textsc{D\'epartement de math\'ematiques, Universit\'e de Sherbrooke\\ Sherbrooke, QC J1K 2R1, Canada} \par
\textsc{Mathematics Department, Bishop's University\\ Sherbrooke, QC J1M 1Z7, Canada} \par
  \textit{E-mail address}: 
  \texttt{Thomas.Brustle@USherbrooke.ca} \par
  \addvspace{\medskipamount}
  \textsc{D\'epartement de math\'ematiques, Universit\'e de Sherbrooke\\ Sherbrooke, QC J1K 2R1, Canada} \par
  \textit{E-mail address}: 
  \texttt{Justin.Desrochers@USherbrooke.ca} \par
  \addvspace{\medskipamount}
  \textsc{Department of Mathematics and Statistics, Queen's University\\ Kingston, ON K7L 3N6, Canada} \par  
  \textit{E-mail address}: 
  \texttt{Samuel.Leblanc@QueensU.ca}
}}

\begin{document}

\begin{abstract}
Let $\cC$ be a small, connected category with finite hom-sets.  
We show that if the embedding of a connected subcategory $\cJ$ is both initial and final, then the restriction of any $\cC$-module along $\cJ$ preserves the generalized rank---or equivalently, the multiplicity of the “entire” interval modules for $\cC$ and $\cJ$. 
Conversely, we prove that this property characterizes initial and final embeddings when both $\cC$ and $\cJ$ are posets satisfying certain mild constraints and the embedding is full.  
For $\cC$ a poset under these conditions, we describe the minimal full subposet whose embedding is initial or final. 
This generalizes an observation made by Dey and Lesnick. 
We also extend a result of Kinser on the generalized rank invariant to small categories.
\end{abstract}

\maketitle

\section{Introduction}

\subsection{Motivation and Previous Results}

In topological data analysis, more specifically in multiparameter persistence theory (see, e.g., \cite{oudot2015, chazalmichel2021, botnanlesnick22,bauerbrustlescoccola2024} for an introduction), one studies \emph{persistence modules} $M$ over an $n$-dimensional grid $\cP$, typically modeled as a product of $n$ totally ordered sets. 
The representation theory of such posets is generally \emph{wild} and it is difficult to interpret all indecomposable summands of a given $\cP$-persistence module $M$. 
A common strategy in the literature is to focus on indecomposable summands that are easier to interpret, like the interval module $ \K_\cP$ which is one-dimensional over a fixed field $\K$ at every vertex of $\cP$. 
For a connected poset $\cP$, one can study the following two fundamental questions. 

\begin{enumerate}
    \item The interval module $ \K_\cP$ is indecomposable. 
    It is therefore natural to ask wether it is a direct summand of $M$ and if so, what is the multiplicity \( \mult{\K_\cP}{M} \) of \( \K_\cP \) in \( M \).
    
    \item Generalizing the rank invariant introduced by Carlsson and Zomorodian in \cite{carlssonzomorodian09}, one can take the limit and colimit of $M$ viewed as a diagram over $\cP$ and ask what is the rank \( \mathrm{rk}\Psi_M \) of the canonical map from the limit to the colimit (see \cite{kimmemoli2021}) $\Psi_M : \lim M \to \colim M$. 
\end{enumerate}
Remarkably, the numbers \( \mult{\K_\cP}{M} \) and \( \mathrm{rk}\Psi_M \) coincide. This equality was shown by Chambers and Letscher \cite{chambersletscher2019} in the context of persistence modules, and earlier by Kinser \cite{kinser2008} for representations of tree quivers. 

A unifying perspective in all of these cases is to view a persistence module \( M \) as a functor from a small category---either a poset \( \cP \) or a quiver \( Q \)---to the category of vector spaces over a fixed field \( \K \). Interestingly, the equality of \( \mult{\K_\cP}{M} \) and \( \mathrm{rk}\Psi_M \) relies only on categorical properties and not on specific features of posets or quivers. 
This suggests that rank and multiplicity should agree whenever $M$ is a diagram indexed by a small connected category, as we show in \cref{thm:generalized-rank-general}. 

For efficient computation of rank \( \mathrm{rk}\Psi_M \), it is natural to consider the following problem. 
\begin{enumerate}
\setcounter{enumi}{2}
    \item Find a subposet \( \cS \subseteq \cP \) such that, for any $\cP$-persistence module, the rank of the limit-to-colimit map for $M$ can be recovered from the limit-to-colimit map of the restriction $M|_\cS$. 
\end{enumerate}
In the language of categories, a connected subposet $\cS$ satisfies (3) whenever there is an inclusion functor $F:\cS\hookrightarrow \cP$  which is both \emph{final} and \emph{initial} (\cref{subsec-finial-initial}). As seen in our \cref{thm-colimit-preserving-restriction}, a poset functor $F$ is final (dually, initial) if and only if $M$ and $M|_\cS = MF$ have the same colimit (dually, limit) for every persistence module $M$ over $\cP$. 

For a two-dimensional grid posets, problem (3) was studied by Dey-Kim-M\'emoli in \cite{deykimmemoli2023}. 
They proved that if $\cP$ is a finite interval in $\mathbb{Z}^2$, then one can restrict to some zigzag subposet to compute the rank of the limit-to-colimit map. 
The subposet constructed in \cite{deykimmemoli2023} is initial and final (i.e., its embedding is an initial and final functor). 
More generally, Botnan-Oppermann-Oudot \cite[Proposition 3.9]{boo} gave a sufficient condition for restriction along a subposet $\cS \subseteq \bbR^n$ to preserve the \textit{generalized rank invariant}. 
This condition amounts to asking that the subposet is initial and final. 
Both of these results agree with our \cref{thm-colimit-preserving-restriction}. 

Given these results, we naturally ask what is the minimal (with respect to inclusion) full subposet such that its embedding is final, and dually, initial. 
In general, such a minimal full subposet need not exist, see, e.g., \cref{fig-exampleposets} and \cref{ex-no-min-fin-sub}. 

Under some assumptions on the poset $\cP$, we construct, in \cref{thm-minimal-final-subposet} and \cref{rmk-dual-minimal-final}, the minimal final and initial full subposets $ \sinf^{fin}$ and $ \sinf^{init}$.
In \cref{construction-M-rk}, we connect these two subposets to find a subposet $\cM{\rk}$ which is a solution to problem (3). 
This connected subposet $\cM{\rk}$ is in general neither unique nor a full subposet. In the case that $\cP$ is an interval in $\bbZ^2$, such as in \cref{ex-explicitcomputation}, $ \sinf^{fin}$ and $ \sinf^{init}$ are the upper and lower fences of \cite[Definition 3.5]{deykimmemoli2023} and certain choices of $\cM{\rk}$ agree with the notion of \textit{boundary cap} in \cite[Definition 3.7]{deykimmemoli2023}.

\subsection{Contributions}\label{sec-main-results}
We now present our main contributions, formulated in the most natural generality. 

In \cref{thm:generalized-rank-general}, we extend the relationship between problems (1) and (2) above to $\cC$-modules. 
\cref{prp-multiplicity-preserving-restriction} then follows by combining \cref{thm:generalized-rank-general} with a classical result in category theory: the functor $F : \cJ \to \cC$ is final if and only if $M F$ and $M$ have the same colimit for all functors $M : \cC \to \cD$ (see \cref{thm-final-preserves-colimit}).

\begin{letterprp}\label{prp-multiplicity-preserving-restriction}\textbf{\emph{(Multiplicity Preserving Restriction).}} 
Let $\cJ$ and $\cC$ be small connected categories.
If $F : \cJ \to \cC$ is final and initial, then $\mult{\intMod{\cC}}{M} = \mult{\intMod{\cJ}}{MF}$ for all functors $M : \cC \to \Vect_\K$. 
\end{letterprp}

Our next result shows that, to check whether a functor is final, it suffices to consider only pointwise finite-dimensional $\cC$-modules.

\begin{letterthm}\label{thm-colimit-preserving-restriction}\textbf{\emph{((Co)limit Preserving Restriction).}} 
Let $\cJ$ and $\cC$ be small categories and suppose that $\cC$ has finite hom-sets.
A functor $F: \cJ \to \cC$ is final (resp. initial) if and only if $\colim MF \cong \colim M$ (resp. $\lim MF \cong \lim M$) for all functors $M : \cC \to \vect_\K$. 
\end{letterthm}

We also obtain the corresponding result for (potentially) infinite-dimensional $\cC$-modules. 

\begin{cor}\label{cor-colimit-perserving-restriction-Mod}
Let $\cJ$ and $\cC$ be small categories and suppose that $\cC$ has finite hom-sets. 
A functor $F: \cJ \to \cC$ is final (resp. initial) if and only if $\colim MF \cong \colim M$ (resp. $\lim MF \cong \lim M$) for all functors $M : \cC \to \Vect_\K$. 
\end{cor}

We are interested in finding the minimal subposet such that the embedding is final or initial. 
We will prove that such a subposet exists, and construct it, under the following assumptions. Here we denote by $\down \cS$ the downset of a subset $\cS$ of $\cP$, which is given by all points that are smaller than or equal to some point in $\cS$; likewise for the upset $\up \cS$. 

\begin{asn}\label{asn-poset-fin}
    A nonempty poset $\cP$ satisfies this assumption if 
    \begin{enumerate}
        \item\label{asn-poset-fin-downset} $\cP = \down \cS$ for some finite subposet $\cS \subseteq \cP$, and
        
        \item\label{asn-poset-fin-maximality} For all subposets $\cA, \cB \subseteq \cP$ such that $\cA \cap \cB \neq \emptyset$ and which are of the form $\down p$ or $\down p \cap \up q$ for some $p,q \in \cP$, we have $\max(\cA\cap\cB) \neq \emptyset$. 
    \end{enumerate}
\end{asn}

\begin{asn}\label{asn-poset-fin-dual}\textbf{(Dual of \cref{asn-poset-fin})}
    A nonempty poset $\cP$ satisfies this assumption if 
    \begin{enumerate}
        \item $\cP = \up \cS$ for some finite subposet $\cS \subseteq \cP$, and
        
        \item\label{asn-poset-fin-minimality} For all subposets $\cA, \cB \subseteq \cP$ such that $\cA \cap \cB \neq \emptyset$ and which are of the form $\up p$ or $\up p \cap \down q$ for some $p,q \in \cP$, we have $\min(\cA\cap\cB) \neq \emptyset$. 
    \end{enumerate}
\end{asn}

We illustrate in Figure \ref{fig-exampleposets} some examples where these assumptions are satisfied or not.

\begin{figure}
    \centering
    \begin{subfigure}[t]{0.48\textwidth}
        \begin{center}
        \begin{tikzpicture}
        \fill[fill=blue!60] (0,2) -- (1,2) -- (1,1) -- (2,1) -- (2,0) -- (4,0) -- (4,3) -- (3,3) -- (3,4) -- (0,4) -- (0,2);
        \draw[] (0,2) -- (1,2) -- (1,1) -- (2,1) -- (2,0) -- (4,0) -- (4,3) -- (3,3) -- (3,4) -- (0,4) -- (0,2);
        \fill[fill=black] (0,2) circle (2.5pt) node[left] {$s_1$};
        \fill[fill=black] (1,1) circle (2.5pt) node[left] {$s_2$};
        \fill[fill=black] (2,0) circle (2.5pt) node[left] {$s_3$};
        \fill[fill=black] (4,3) circle (2.5pt) node[right] {$t_2$};
        \fill[fill=black] (3,4) circle (2.5pt) node[right] {$t_1$};
        \end{tikzpicture}
        \end{center}
        \caption{$\cP = \up\{s_1,s_2,s_3\} \cap \down \{t_1, t_2\} \subset \bbR^2$ satisfies \cref{asn-poset-fin} and \cref{asn-poset-fin-dual}.}
        \label{fig-exampleposets-inR2}
    \end{subfigure}
    \hfill
    \begin{subfigure}[t]{0.48\textwidth}
        \begin{center}
        \tdplotsetmaincoords{60}{120} % view angles: (elevation, azimuth)
        \begin{tikzpicture}[line join=round]
        
          % Cube size
          \def\L{3}
        
          % Define the 8 cube vertices
          \coordinate (O) at (0,0,0);
          \coordinate (A) at (\L,0,0);
          \coordinate (B) at (\L,\L,0);
          \coordinate (C) at (0,\L,0);
          \coordinate (D) at (0,0,\L);
          \coordinate (E) at (\L,0,\L);
          \coordinate (F) at (\L,\L,\L);
          \coordinate (G) at (0,\L,\L);
        
          % Draw faces
          \fill[blue!60,opacity=0.5] (O) -- (A) -- (B) -- (C) -- cycle; % bottom
          \fill[blue!60,opacity=0.5] (O) -- (A) -- (E) -- (D) -- cycle; % front
          \fill[blue!60,opacity=0.5] (A) -- (B) -- (F) -- (E) -- cycle; % right
          \fill[blue!60,opacity=0.5] (B) -- (C) -- (G) -- (F) -- cycle; % back
          \fill[blue!60,opacity=0.5] (C) -- (O) -- (D) -- (G) -- cycle; % left
          \fill[blue!60,opacity=0.5] (D) -- (E) -- (F) -- (G) -- cycle; % top
        
          % Draw cube edges
          \draw[dashed] (C) -- (O) -- (A);
          \draw[] (C) -- (B) -- (A);
          \draw[] (D) -- (E) -- (F) -- (G) -- cycle;
          \draw[dashed] (O) -- (D);
          \draw[] (A) -- (E);
          \draw[] (B) -- (F);
          \draw[] (C) -- (G);

          \fill[fill=black] (D) circle (2.5pt) node[left] {$s$};
          \fill[fill=black] (B) circle (2.5pt) node[right] {$t$};
        \end{tikzpicture}
        \end{center}
        \caption{$\cP = \up s \cap \down t \subset \bbR^3$ satisfies \cref{asn-poset-fin} and \cref{asn-poset-fin-dual}.}
        \label{fig-exampleposets-cube}
    \end{subfigure}
    \begin{subfigure}[t]{0.48\textwidth}
        \begin{center}
        \begin{tikzpicture}
        \fill[fill=red!50] (0,0) -- (3.5,0) -- (0,3.5) -- (0,0);
        \draw[] (0,0) -- (3.5,0) -- (0,3.5) -- (0,0);
        \fill[fill=black] (0,0) circle (2.5pt) node[left] {$s$};
        \draw[ultra thick] (3.5,0) -- (0,3.5);
        \node[] at (2.85, 1) {$T$};
        \end{tikzpicture}
        \end{center}
        \caption{$\cP = \up s \cap \down T \subset \bbR^2$ doesn't satisfy \cref{asn-poset-fin} (\ref{asn-poset-fin-downset}), since $T$ is not contained in the downset of a finite subset of $\cP$.}
        \label{fig-exampleposets-failDownset}
    \end{subfigure}
    \hfill
    \begin{subfigure}[t]{0.48\textwidth}
        \begin{center}
        \begin{tikzpicture}
        \fill[fill=red!50] (0,0) -- (4,0) -- (4,2) -- (2.5,2) -- (2.5, 3.5) -- (0,3.5) -- (0,0);
        \draw[] (0,0) -- (4,0) -- (4,2) -- (2.5,2) -- (2.5, 3.5) -- (0,3.5) -- (0,0);
        \fill[fill=black] (0,0) circle (2.5pt) node[left] {$s$};
        \fill[fill=black] (2.5,3.5) circle (2.5pt) node[right] {$t_1$};
        \fill[fill=black] (4,2) circle (2.5pt) node[right] {$t_2$};
        \fill[fill=white] (2.5,2) circle (2.5pt);
        \draw[thick] (2.5,2) circle (2.5pt) node[above right] {$r$};
        \end{tikzpicture}
        \end{center}
        \caption{$\cP = \up s \cap \down \{t_1, t_2\}\setminus \{r\} \subset \bbR^2$ doesn't satisfy \cref{asn-poset-fin} (\ref{asn-poset-fin-maximality}), since the subposet $\down t_1 \cap \down t_2$ is non-empty, but $\max (\down t_1 \cap \down t_2) = \emptyset$.}
        \label{fig-exampleposets-failIntersection}
    \end{subfigure}
    \caption{The top row displays posets satisfying \cref{asn-poset-fin} and \cref{asn-poset-fin-dual}. The posets in the bottom row do not satisfy \cref{asn-poset-fin}.
    }
    \label{fig-exampleposets}
\end{figure}

Define recursively the full subposets $\cS_0 \defeq \max\cP$ and
\begin{equation*}
\cS_n \defeq \cS_{n-1} \sqcup \max\{p \in \cP : \up p \cap \cS_{n-1} \text{ is not connected}\}
\end{equation*}
for $n \geq 1$. 
We then consider $\sinf = \bigcup_{n=0}^\infty S_n$. 

\begin{letterthm}\label{thm-minimal-final-subposet}\textbf{\emph{(Minimal Final Subposet).}} 
Let $\sinf$ be the poset constructed above and suppose $\cP$ satisfies \cref{asn-poset-fin}. 
\begin{enumerate}
    \item The full embedding $F : \sinf \hookrightarrow \cP$ is final. 
    \item If $G : \cQ \hookrightarrow \cP$ is a final poset embedding, then $\Obj\sinf \subseteq \Obj\cQ$. 
    \item There is an $\ell \leq \#\max \cP - 1$ such that $\cS_\ell = \cS_{\ell+1} = \cdots = \cS_\infty$. Furthermore, this bound is sharp.
\end{enumerate}
\end{letterthm}

\begin{rmk}\label{rmk-dual-minimal-final}
If we replace \cref{asn-poset-fin} by \cref{asn-poset-fin-dual}, $\max$ by $\min$, and $\up p$ by $\down p$, we obtain the dual of \cref{thm-minimal-final-subposet}, which gives the minimal subposet such that its embedding is initial. 
\end{rmk}

\begin{rmk}
Let $\cP$ be a poset such that $\up p$ (dually, $\down p$) is finite for all $p\in \cP$.
    In this case, parts (1) and (2) of \cref{thm-minimal-final-subposet} were first obtained by Dey and Lesnick \cite{mikestalk-october}.  
    Note that $\cP$ satisfies \cref{asn-poset-fin}(\ref{asn-poset-fin-maximality}) (dually, \cref{asn-poset-fin-dual}(\ref{asn-poset-fin-minimality})). 
    In this case, Dey and Lesnick construct $\cS_\infty$ as the \textit{final hull} of $\cP$: the full subposet $\{p \in \cP : \up p \setminus \{p\} \text{ is not connected}\}$. 
    As noted in \cref{subsec-related-work}, this result was recently refined to construct a nonfull subposet that also minimizes the number of relations \cite[Thm.~3.5]{deylesnick2026}). 
\end{rmk}

\begin{rmk}
Every finite poset satisfies \cref{asn-poset-fin} and \cref{asn-poset-fin-dual}. 
More generally, any finitely presentable subposet of $\bbR^n$ (in the sense of \cite[Section 2.6]{BDHS25})  satisfies \cref{asn-poset-fin-dual} (see \cite[Lemma 2.1]{bauerscoccola25}). 
For instance, any finite interval in $\mathbb{Z}^2$ satisfies these assumptions, allowing us to recover the zigzag posets of \cite{deykimmemoli2023}. 
We illustrate a specific case of their result in \cref{ex-explicitcomputation}. 
%Finally, the downset (dually, upset) of a finite set of points in $\bbZ^n$ satisfies \cref{asn-poset-fin} (dually, \cref{asn-poset-fin-dual}). 
\end{rmk}

\begin{rmk}
Even when $\cP$ is a continuous poset, the subposet $\sinf$ can be finite. This occurs, for example, in \cref{fig-exampleposets-inR2} and \cref{fig-exampleposets-cube}.
\end{rmk}

\begin{rmk}\label{rmk-drop-assumption}
If we drop any of the parts of \cref{asn-poset-fin} or \cref{asn-poset-fin-dual}, it is easy to find an example of a poset that does not admit a minimal final or initial subposet.
For instance, this is the case for \cref{fig-exampleposets-failDownset} and \cref{fig-exampleposets-failIntersection} due to a similar reason as the one given in \cref{ex-no-min-fin-sub}.
\end{rmk}

Under \cref{asn-poset-fin} or \cref{asn-poset-fin-dual}, we can show a partial converse of \cref{prp-multiplicity-preserving-restriction}: 
\begin{letterprp}\label{prp-converseB}
Suppose that $F:\cS \to \cP$ is a full poset morphism between connected posets and that $\cP$ satisfies \cref{asn-poset-fin} (resp. \cref{asn-poset-fin-dual}). If $\mult{\intMod{\cP}}{M} = \mult{\intMod{\cS}}{MF}$ for all $M \in \mod\cP$, then $F$ is final (resp. initial). 
\end{letterprp}

Note that the above proposition is false if we do not assume that the poset morphism is full. 
We give a counter-example in \cref{ex-counter-ex-converse-B}.

\begin{rmk}
\cref{prp-converseB} holds when $F : \sinf \hookrightarrow \cP$ is the embedding of the minimal final (or initial) subposet constructed in \cref{thm-minimal-final-subposet}. 
\end{rmk}

We summarize how to use our results to compute the multiplicity of the entire interval module in:

\begin{lettercon}\label{construction-M-rk}
Let $\cP$ be a poset satisfying \cref{asn-poset-fin} and \cref{asn-poset-fin-dual}.
Using \cref{thm-minimal-final-subposet} and \cref{rmk-dual-minimal-final}, we construct a poset embedding $F:\cM{\rk} \hookrightarrow \cP$ such that
\begin{equation*}
\mult{\intMod{\cP}}{M} = \mult{\intMod{\cM{\rk}}}{MF}
\end{equation*}
for all $M \in \mod\cP$. 
Moreover, $\cM{\rk}$ is such that if $G : \cQ \hookrightarrow \cP$ is a final and initial poset embedding, then $\Obj\cM{\rk} \subseteq \Obj\cQ$. 

Our strategy is to apply the following steps:
\begin{enumerate}
    \item\label{step1} Find the minimal subposet $\sinf^{fin}$ such that its embedding is final (see \cref{thm-minimal-final-subposet}). 
    
    \item\label{step2} Find the minimal subposet $\sinf^{init}$ such that its embedding is initial (see \cref{rmk-dual-minimal-final}). 
    
    \item\label{step3} Consider the union $\cU \defeq \sinf^{fin} \cup \sinf^{init}$. 
    If $F:\cQ \hookrightarrow \cP$ is a final and initial poset embedding, then by \cref{thm-minimal-final-subposet} (2) and its dual, we have $\Obj\cU \subseteq \Obj\cQ$.
    
    \item\label{step4} To use \cref{thm:generalized-rank-general}, one needs to have a connected poset. 
    So, if $\cU$ is disconnected, we add an order relation (from $\cP$) between an element of $\sinf^{init}$ and one of $\sinf^{fin}$ to obtain $\cM{\rk}$. 
    In other words, $\cM{\rk}$ is obtained from $\cU$ by turning it into a connected (but not necessarily full) subposet of $\cP$. 
    
    \item\label{step5} At this point, one has a connected subposet $\cM{\rk}$ such that the poset inclusion functor $F : \cM{\rk} \hookrightarrow \cP$ preserves the generalized rank, or equivalently, by \cref{thm:generalized-rank-general}, the multiplicity of the entire interval module. 
    We prove this in \cref{prp-constructionWorks}. 
\end{enumerate}
\end{lettercon}

One could change Part \ref{step4} slightly by allowing one to add more than a single relation. 
This would not change the validity of \cref{construction-M-rk}. 
If one adds all possible relations, then we get the following.

\begin{rmk}\label{rmk-fullSubposetRank}
    Let $\cM{full}$ be the full subcategory of $\cP$ generated by the objects in $\sinf^{init} \cup \sinf^{fin}$. 
    By \cref{prp-converseB} and \cref{prp-constructionWorks}, $\cM{full}$ is final, initial, and minimal in the sense that any final and initial subposet $\cQ \subseteq \cP$ satisfies $\Obj\cM{full} \subseteq \Obj\cQ$. 
\end{rmk}

In general, if one doesn't consider $\cM{full}$ but simply $\cM{\rk}$, the conclusion of \cref{rmk-fullSubposetRank} is false. 
This is what we show in the next example.

\begin{ex}
    Take $\cP$ to be the interval in \cref{fig-exampleposets-inR2}. Then 
    \begin{equation*}
    \sinf^{init} = \{s_1, s_2, s_3, s_1 \vee s_2, s_2 \vee s_3 \}\, \text{ and } \,\sinf^{fin} = \{t_1, t_2, t_1 \wedge t_2 \},
    \end{equation*}
    where we denote the join by $\vee$ and the meet by $\wedge$. 
    Following \cref{construction-M-rk}, construct $\cM{\rk}$ by adding the morphism $s_2 \leq t_2$ to $\sinf^{init} \cup \sinf^{fin}$. 
    By \cref{prp-constructionWorks}, the embedding $F: \cM{rk} \hookrightarrow \cP$ preserves the generalized rank.
    However, $F^{-1}(\up s_1) = \{s_1, s_1 \vee s_2 \} \cup \sinf^{fin}$ is not connected in $\cM{\rk}$. 
    Then $F$ is not final, by \cref{lem-comma-cat-poset}. 
\end{ex}

We conclude this subsection with an explicit example using our aforementioned results. 

\begin{ex}\label{ex-explicitcomputation}
We apply \cref{construction-M-rk} to compute $\mult{\intMod{\cP}}{M}$, where $\K = \bbR$ and $M$ is the following  
$\cP$-module:
\begin{equation*}
M = 
\begin{tikzcd}[ampersand replacement=\&, sep=4.2em] 
% Top row 
\color{cyan}{\bbR^2} \arrow[r, cyan,  "{\lbracket\begin{smallmatrix}1 & 0\\1 & 1\end{smallmatrix}\rbracket}" cyan] \& \color{cyan}{\bbR^2} \arrow[r, cyan,  "{\lbracket\begin{smallmatrix} 1 & -1\\0 & 1\end{smallmatrix}\rbracket}" cyan] \& {\color{cyan}\bbR^2} \arrow[r, cyan,   "{\lbracket\begin{smallmatrix}-1 & 1\\0 & 0\\1 & 0\end{smallmatrix}\rbracket}" cyan] \& {\color{cyan}\bbR^3} \arrow[r, cyan,   "{\lbracket\begin{smallmatrix}1 & 0 & 0\\0 & -1 & 0\end{smallmatrix}\rbracket}" cyan] \& {\color{red}\bbR^2}\\ 
% Middle row 
{\color{red}\bbR^3} \arrow[r, red,   "{\lbracket\begin{smallmatrix}1 & 0 & 0\\0 & 1 & 0\\0 & 0 & 1\end{smallmatrix}\rbracket}" red] \arrow[u, cyan,  "{\lbracket\begin{smallmatrix}1 & 1 & 0\\-1 & 0 & -1\end{smallmatrix}\rbracket}" cyan] \& {\color{red}\bbR^3} \arrow[r, blue,   "{\lbracket\begin{smallmatrix}1 & 0 & 1\\0 & 1 & -1\end{smallmatrix}\rbracket}" blue] \arrow[u,   "{\lbracket\begin{smallmatrix}1 & 1 & 0\\0 & 1 & -1\end{smallmatrix}\rbracket}"] \& {\color{blue}\bbR^2} \arrow[r,  blue, "{\lbracket\begin{smallmatrix}-1 & 1\\0 & 1\end{smallmatrix}\rbracket}" blue] \arrow[u,   "{\lbracket\begin{smallmatrix}1 & 0\\0 & 1\end{smallmatrix}\rbracket}"] \& \color{blue}{\bbR^2} \arrow[r, blue, "{\lbracket\begin{smallmatrix}2 & 1\\0 & 1\\1 & 1\end{smallmatrix}\rbracket}" blue] \arrow[u,   "{\lbracket\begin{smallmatrix}1 & 0\\0 & 0\\-1 & 1\end{smallmatrix}\rbracket}"] \& \color{blue}{\bbR^3} \arrow[u, blue, "{\lbracket\begin{smallmatrix}1 & 0 & -1\\1 & 1 & -2\end{smallmatrix}\rbracket}" blue]\\ 
% Bottom row 
\& {\color{red}\bbR^2} \arrow[r,  "{\lbracket\begin{smallmatrix}1 & -1\\1 & 1\end{smallmatrix}\rbracket}"] \arrow[u, red,   "{\lbracket\begin{smallmatrix}1 & -1\\0 & 2\\1 & 1\end{smallmatrix}\rbracket}" red] \& \bbR^2 \arrow[r,   "{\lbracket\begin{smallmatrix}1 & 0\\0 & 1\\-1 & 1\end{smallmatrix}\rbracket}"] \arrow[u,   "{\lbracket\begin{smallmatrix}1 & 1\\-1 & 0\end{smallmatrix}\rbracket}"] \& \bbR^3 \arrow[r,   "{\lbracket\begin{smallmatrix}0 & -1 & 0\\-1 & 0 & 0\end{smallmatrix}\rbracket}"] \arrow[u,   "{\lbracket\begin{smallmatrix}-2 & -1 & 0\\-1 & 0 & 0\end{smallmatrix}\rbracket}"] \& \bbR^2 \arrow[u,   "{\lbracket\begin{smallmatrix}2 & 5\\0 & 1\\1 & 3\end{smallmatrix}\rbracket}"] 
\end{tikzcd}
\end{equation*}

We follow the steps outlined in \cref{construction-M-rk} to compute $\mult{\intMod{\cP}}{M}$. 
For Step (\ref{step1}), we see that $\sinf^{fin}$ is the unique maximal element of $\cP$. 
For Step (\ref{step2}), we see that $\sinf^{init}$ is the full subposet with elements $\min\cP \sqcup \min\{p \in \cP : \down p \cap \min \cP \text{ is not connected}\}$, so $\sinf^{init}$ contains the two minima and their join in $\cP$. 
For Step (\ref{step3}), write $\cU \defeq \sinf^{fin} \sqcup \sinf^{init}$ (in \textcolor{red}{red}) for the minimal subposet such that its embedding is both final and initial. 

There is no unique choice of relation to add at Step (\ref{step4}). 
For instance, one could connect $\cU$ via (the concatenation of) the \textcolor{blue}{blue} path to obtain $\cM{b} = \cM{\rk}$ or via (the concatenation of) the \textcolor{cyan}{cyan} path to obtain $\cM{c} = \cM{\rk}$. 
Let $F_b : \cM{b} \hookrightarrow \cP$ and $F_c : \cM{c} \hookrightarrow \cP$ be the corresponding embedding functors. 
We get the $\cM{b}$-module and the $\cM{c}$-module
\begin{equation*}
MF_b = 
\begin{tikzcd}[ampersand replacement=\&, sep=2.9em]
\& \& \color{red}{\bbR^2}\\
\color{red}{\bbR^3} \arrow[r, red, "\id" red] \& \color{red}{\bbR^3} \arrow[ur, blue, "{\lbracket\begin{smallmatrix}-1 & 1 & -2\\0 & 0 & 0\end{smallmatrix}\rbracket}" blue] \&\\
\& \color{red}{\bbR^2} \arrow[u, red,  "{\lbracket\begin{smallmatrix}1 & -1\\0 & 2\\1 & 1\end{smallmatrix}\rbracket}" red]
\end{tikzcd}
\text{ and }
MF_c = 
\begin{tikzcd}[ampersand replacement=\&, sep=2.9em]
\& \& \color{red}{\bbR^2}\\
\color{red}{\bbR^3} \arrow[urr, cyan, "{\lbracket\begin{smallmatrix}-1 & 1 & -2 \\ 0 & 0 & 0 \end{smallmatrix}\rbracket}" cyan] \arrow[r, red, "\id" red] \& \color{red}{\bbR^3}\\
\& \color{red}{\bbR^2} \arrow[u, red,  "{\lbracket\begin{smallmatrix}1 & -1\\0 & 2\\1 & 1\end{smallmatrix}\rbracket}" red]
\end{tikzcd}
\end{equation*}
depending on the path chosen. 
In the language of quivers, one gets in the first case a representation of type $\mathbb{D}_4$ and in the latter, of type $\mathbb{A}_4$, like obtained in \cite{deykimmemoli2023}. 
Their representation theory is extremely well understood \cite{Gab72, ASS06}. 
In particular, $\cM{\rk}$ is of finite representation type (i.e., there are only finitely many isomorphism classes of indecomposable modules), as opposed to $\cP$, which is of infinite type \cite{EscolarHiraoka2015}. 
At Step (\ref{step5}), one can answer the original question. 
We know from \cref{prp-multiplicity-preserving-restriction} that 
\begin{equation*}
\mult{\intMod{\cP}}{M} = \mult{\intMod{\cM{b}}}{MF_b} = \mult{\intMod{\cM{c}}}{MF_c}
\end{equation*}
and we can see that $\mult{\intMod{\cM{b}}}{MF_b} = \mult{\intMod{\cM{c}}}{MF_c} = 1$. 
We can also get this from an explicit computation. 
By \cref{thm:generalized-rank-general}, we know that $\mult{\intMod{\cM{b}}}{MF_b} = \rk MF_b$. 
We can compute that 
\begin{equation*}
\rk MF_b = \rk\left(\lbracket\begin{smallmatrix}-1 & 1 & -2\\0 & 0 & 0\end{smallmatrix}\rbracket\lbracket\begin{smallmatrix}1 & -1\\0 & 2\\1 & 1\end{smallmatrix}\rbracket\right) = 1
\end{equation*}
which implies that $\mult{\intMod{\cM{b}}}{MF_b} = 1 = \mult{\intMod{\cP}}{M}$. 
\end{ex}

\subsection{Related Work}\label{subsec-related-work}

Efficient computation of the generalized rank invariant has received increased attention in recent years.  
For a persistence module $M$ indexed by a finite poset, \cite{deyxin2025} develops a \textit{folding/unfolding construction} of a zigzag persistence module with the same generalized rank as $M$.  
The generalized rank of a zigzag module is equivalent to its barcode, and the techniques of \cite{deykimmemoli2023} and \cite{deyxin2025} enable the use of fast barcode algorithms for zigzag persistence \cite{DeyHou22,DeyHouMorozov25}.

For a poset $\cP$ with finite upsets, the minimal final full subposet was first constructed by Dey and Lesnick \cite{mikestalk-october}.  
Our \cref{thm-minimal-final-subposet} generalizes this result to other posets. 
In a recent talk \cite{mikestalk}, Dey and Lesnick proposed a stronger definition of \textit{minimal}, which accounts for the number of relations in the subposet.  
This leads to a smaller final subposet that is not unique, yet still satisfies a nonfull version of parts (1) and (2) of \cref{thm-minimal-final-subposet} when $\cP$ has finite upsets.  
They also provide algorithms and complexity bounds for computing such subposets \cite{deylesnick2026}.

Several other generalizations of the rank invariant and multiplicity have appeared in the literature. \cite{kimmemoli2021} considered the \textit{interval-generalized rank invariant}, computing the $\rk M|_\cI$ along each interval $\cI$ of a poset $\cP$. Since $\cP$ is an interval in itself, the \textit{interval-generalized rank invariant} distinguishes more modules. They also show that the interval rank invariant is stable in the \textit{erosion distance}. 

The interval-generalized rank invariant is realized as the \textit{total compressed multiplicity} in \cite{asashibaetal2023}. 
The \textit{compressed multiplicity} computes the multiplicity \( \mult{\K_\cQ}{M F_\cQ} \) of the restriction along a (not necessarily injective) poset morphism $F_\cQ:\cQ\to \cP$. Several different compressed multiplicities have been studied by \cite{asashibaetal2023, asashibagauthierliu2024, asashibaliu2025}. 
Most recently, \cite{asashibaliu2025} provided a formula for computing the interval multiplicities of a persistence module. They show that the multiplicity of an interval summand can be computed by restricting to another poset which \textit{essentially covers} $\cP$.

\subsection{Organization of the Paper}

In \cref{sec-background}, we recall background on categories (\cref{subsec-categories}), posets (\cref{subset-posets}), $\cC$-modules (\cref{subsec-persistence-modules}), limits and colimits (\cref{section::lim-colim}), and final and initial functors (\cref{subsec-finial-initial}). 
In \cref{subsec-generalized-rank}, we prove a generalization of \cref{thm-rk-decompose}, and obtain \cref{prp-multiplicity-preserving-restriction} as a consequence. 
The proof of \cref{thm-colimit-preserving-restriction} and \cref{cor-colimit-perserving-restriction-Mod} are in \cref{sec-proofThmA}. Finally, the proof of \cref{thm-minimal-final-subposet} is in \cref{sec-proofThmB}. 
Our main results are stated in \cref{sec-main-results}. 

\subsection{Acknowledgements}
The authors would like to thank Laurianne Baril, Luis Scoccola, and Greg Stevenson for useful discussions. 
We also thank Tamal K. Dey and Michael Lesnick for pointing out a problem in an earlier version of our \cref{construction-M-rk}. 
We are grateful to the anonymous referee for careful reading and valuable suggestions.
TB was supported by Bishop’s University, Universit\'e de Sherbrooke and
NSERC Discovery Grant RGPIN-2025-05047.
JD was supported by Universit\'e de Sherbrooke and the Bourse d'Excellence d'Institut des Sciences Math\'ematiques. 
Most of this work was done while SL was at the Universit\'e de Sherbrooke. 
Some of our results were presented by SL at the AARMS-CMS Student Poster Session of the 2025 CMS Summer Meeting. 

\section{Background}\label{sec-background}

\subsection{Categories}\label{subsec-categories}

Let $\cC$ be a category. 
We write $\Obj\cC$ for the class of objects of $\cC$ and for any $c, c' \in \Obj\cC$, we let $\Hom_\cC(c,c')$ denote the class of morphisms from $c$ to $c'$. 
To abbreviate, we write $c \in \cC$ to mean $c \in \Obj\cC$ when it is clear from the context. 
We denote by $\cC^{op}$ the \emph{opposite category} of $\cC$. 
We say that $\cC$ is \emph{small} if $\Obj\cC$ and $\Hom_\cC(c,c')$ are sets for any $c,c' \in \cC$. 
The category $\cC$ is \emph{connected} if $\Obj\cC \neq \emptyset$ and there is a finite zigzag of morphisms $c \rightarrow d_1 \leftarrow d_2 \rightarrow \cdots \leftarrow d_k \rightarrow c'$ between any two objects $c,c' \in \cC$. 
Adopting notation used for posets or directed graphs, we say a subcategory $\cS$ of $\cC$ is \emph{convex} if for all morphisms $f : s \to t$ in $\cS$ that factor as $f = hg$ for some $g : s \to c$ and $h: c \to t$ in $\cC$, one has that $g$ and $h$ (and $c$) are in $\cS$. 
A connected and convex category is said to be an \emph{interval}\footnote{To avoid confusion with the notion of an interval in the combinatorial poset literature, what we call interval here is also referred to as a \emph{spread}, see \cite{BBH24}.}. 
Note that any category is convex as a subcategory of itself and therefore, any connected category is an interval. 

\subsection{Posets}\label{subset-posets}

Let $\cP = (\cP, \leq)$ be a poset which we regard as a category with $\Obj\cP = \cP$ and where $\Hom_\cP(c,c')$ has one element precisely when $c \le c'$, and is empty otherwise. 
This poset is \emph{connected}, \emph{convex}, or an \emph{interval} if it is respectively connected, convex, or an interval as a category. 
A subposet is \emph{full} if it's a full subcategory. 
A functor $F : \cQ \to \cP$ between posets is called a \emph{poset morphism} and it's \emph{full} if, for all $q, q'$ in $\cQ$, we have $q \leq q'$ if and only if $F(q) \leq F(q')$. 
For a subposet $\cS\subseteq \cP$, we denote its inclusion, or \emph{embedding}, by $\cS \hookrightarrow \cP$. 
In particular, if $\cS$ is full, then $\cS \hookrightarrow \cP$ is full. 
If $\cS \subseteq \cP$ is a subposet, we call the \emph{upset} of $\cS$, written $\up \cS$, the full subcategory with objects $\Obj\up \cS \defeq \{p \in \cP : \exists s \in \cS \text{ s.t. }s\leq p\}$. 
Similarly, the \emph{downset} of $\cS$, denoted $\down\cS$, has objects $\Obj\down \cS \defeq \{p \in \cP : \exists s \in \cS \text{ s.t. }p\leq s\}$. 
For a singleton, we write $\up p \defeq \up\{p\}$ and $\down p \defeq \down\{p\}$. 

\subsection{$\cC$-Modules}\label{subsec-persistence-modules}
Throughout, let $\K$ be a field. 
Denote by $\Vect_\K$ the category of $\K$-vector spaces and by $\vect_\K$ the full subcategory of finite-dimensional $\K$-vector spaces. 
Given a small category $\cC$, a $\cC$\emph{-module} is a functor $M : \cC \to \Vect_\K$ and a \emph{pointwise finite-dimensional} $\cC$\emph{-module} is a functor $M : \cC \to \vect_\K$. 
We write $\Mod \cC \defeq \funct{\cC}{\Vect_\K}$ and $\mod \cC \defeq \funct{\cC}{\vect_\K}$ for the categories of such functors. 
Note that $\mod \cC$ is a full subcategory of $\Mod\cC$. 
If it is clear from the context, we simply write $\cC$-module instead of pointwise finite-dimensional $\cC$-module. 
When $\cC = \cP$ is a poset, a $\cC$-module is often referred to in the literature as a $\cP$\emph{-persistence module}.

Let $\cI$ be a convex subcategory of $\cC$. 
We denote by $\intMod{\cI} \in \mod\cC$ the functor
\begin{align*}
    \intMod{\cI}(c) = 
    \begin{cases}
        \K & \text{if }c \in \cI, \\
        0 & \text{otherwise},
    \end{cases}
    \;
     \text{ and }
    \;
     \intMod{\cI}(f : c \to c') = 
    \begin{cases}
        \id & \text{if } f:c\to c' \in \cI, \\
        0 & \text{otherwise.}
    \end{cases}
\end{align*}
If $\cI$ is an interval, we call $\intMod{\cI}$ an \emph{interval module}\footnote{Also denoted $V_\cI$ (\cite{asashibaetal2023} or $M_\cI$ (spread module in \cite{BBH24}), or $\mathbb I_\cI$ (identity representation in \cite{kinser2008})} for $\cI$. 
In particular, we call $\intMod{\cC}$ the \emph{entire interval module of $\cC$}. 

\subsection{Multiplicity}

Let $\cC$ be a small category.
The \emph{direct sum} of $M_1$ and $M_2$ in $\Mod\cC$, denoted $M_1 \oplus M_2$, is the $\cC$-module with $(M_1\oplus M_2)(c) = M_1(c) \oplus M_2(c)$ and $(M_1\oplus M_2)(f:c\to c') = M_1(f:c\to c')\oplus M_2(f : c \to c')$ for all $f : c\to c' \in \cC$. 
If $M \cong M_1 \oplus M_2$ implies that $M_1=0$ or $M_2 = 0$, we say that $M$ is \emph{indecomposable}. 
For instance, every interval module $\intMod{\cI}$ is indecomposable and $\mathrm{End}_\cC (\cI) \cong \K$. 
More generally, every module with \emph{local endormorphism ring} is indecomposable (see \cite[Section~4.8]{Pareigis70}). 

Every $N \in \mod \cC$ can be decomposed as a direct sum of indecomposable objects with local endomorphism rings (\cite[Theorem~1.1]{botnancrawley}), and this decomposition is unique up to isomorphism and permutation of the summands (\cite[Section~4.8]{Pareigis70}). 
For any $U, N \in \mod \cC$, we define the \emph{multiplicity} of $U$ in $N$ to be the largest (cardinal) number $\mult{U}{N}$ such that there exists some $N' \in \mod \cC$ with
\[
N \cong U^{\mult{U}{N}} \oplus N'.
\]
The uniqueness of the decomposition of $N$ ensures that the multiplicity is well-defined for $\mod \cC$.

If $M, L \in \Mod \cC$ and $L$ has a local endomorphism ring, then $\mult{L}{M}$ can be defined in the same way. This is a consequence of \cite[Proposition~1]{warfield69} and the Gabriel--Popescu theorem \cite{GabrielPopescu}.

For our purposes, we will only consider $\mult{\intMod{\cC}}{M}$, for which well-definedness in $\Mod \cC$ follows from \cref{thm:generalized-rank-general}.

\subsection{Limits and Colimits}\label{section::lim-colim}

Let $\cC$ be a small category. The category $\Vect_\K$ is both \emph{complete} and \emph{cocomplete}. 
That is, the limit, $\lim M$, and the colimit, $\colim M$, exists for all $M \in\Mod \cC$.
Moreover, from \cite[Chapter III.4]{maclane} we obtain a triple of adjoint functors 
\begin{equation*}
\begin{tikzcd}
\Mod \cC
  \arrow[r, shift left=1ex, "\colim", bend left]
  \arrow[r, shift right=1ex, "\lim", swap, bend right]
  \arrow[r, phantom, "\bot" description, shift left=2.1ex]
  \arrow[r, phantom, "\bot" description, shift right=2.1ex]
&
\Vect_\K
  \arrow[l, "\Delta" description]
\end{tikzcd}
\end{equation*}
 where the \emph{diagonal functor} $\Delta$ sends an object $V$ in $\Vect_\K$ to the \textit{constant functor} $V_\cC: \cC \to \Vect_\K$ via 
\begin{equation*}
    V_\cC(c) = V
    \;
     \text{ and }
    \;
     V_\cC(g : c \to c') = \id_V
\end{equation*}
and sends a linear map $f : V \to W$ to the natural transformation $(f_c)_{c \in \cC}$ with $f_c = f$ for all $c$. 
Note in particular that $\Delta(\K) = \intMod{\cC}$ is the entire interval module. 
Since the limit is the right adjoint to $\Delta$, the functor $\lim$ is a representable functor, and the representing object is given by $\Delta(\K) = \intMod{\cC}$ (for details, see \cite[Eq. (2) and (3) in Chapter III.4]{maclane}):
\begin{equation}
\lim M \cong \Hom_{\Vect_\K}(\K, \lim M) \cong \Hom_{\Mod\cC}(\K_\cC,M)
\label{equation-lim-hom}
\end{equation}
for all $M \in\Mod \cC$, and dually for colimits one obtains
\begin{equation}
(\colim M)^* = \Hom_{\Vect_\K}(\colim M, \K) \cong \Hom_{\Mod\cC}(M,\K_\cC).
\label{equation-colim-hom}
\end{equation}

In particular, we have the following example.
\begin{ex}\label{ex-computation-colim}
Let $\cL : \setcat \to \Vect_\K$ be the free functor, i.e., the left adjoint to the forgetful functor $\Vect_\K \to \setcat$. 
Let $c $ be an object in a small category $\cC$ and put $M \defeq \cL \circ \Hom_\cC(c,-)$. 
Then by \cref{equation-colim-hom} and the Yoneda Lemma, we get 
\begin{equation}\label{equation-colim}
(\colim M)^* \cong \Hom_{\Mod\cC}(\cL \Hom_\cC(c,-),\K_\cC) \cong \K_\cC(c) =\K
\end{equation}
which implies that $\colim M \cong \K$. 
\end{ex}

\subsection{Final and Initial Functors}\label{subsec-finial-initial}

A functor $F : \cC \to \cD$ is \emph{final} if, for all $d \in \Obj\cD$, the comma category $d/F$ is (non-empty and) connected. 
Dually, $F$ is \emph{initial} if $F/d$ is connected for every object $d$ in $\cD$. 
In fact, $F$ is initial if and only if $F^{op}$ is final.
For further details on final functors, see \cite[Sections II.6 and IX.3]{maclane} or \cite[Proposition 5.2 in Chapter 2]{kashiwara05}. 
The following lemma gives a description of the comma category in case $F : \cS \to \cP$ is a functor between posets.

\begin{lem}\label{lem-comma-cat-poset}
Let $F: \cS \to \cP$ be a poset morphism. 
For all $p \in \cP$, the comma category $p/F$ is isomorphic to $F^{-1}(\up p)$ and similarly, $F/p$ is isomorphic to $F^{-1}(\down p)$. 
\end{lem}

\begin{proof}
We only prove the first part since the second is its dual. 
We first describe $p /F$. 
The objects are
\begin{equation*}
\Obj(p/F) = \{(s, p \leq F(s)) : s \in \cS\}
\end{equation*}
and there is a morphism $(s, p \leq F(s)) \to (t, p \leq F(t))$ in $p/F$ if $s \leq t$ in $\cS$ and $p \leq F(s) \leq F(t)$ in $\cP$. 
However, the first condition implies the last since $F$ is a functor. 
Therefore, we easily see that the categories $F^{-1}(\up p)$ and $p /F$ are isomorphic. 
\end{proof}

\begin{cor}\label{cor-comma-cat-poset-order-emb}
Let $F: \cS \to \cP$ be a full poset morphism. 
For all $p \in \cP$, the comma category $p/F$ is isomorphic to $\up p \cap \im F$ and similarly, $F/p$ is isomorphic to $\down p \cap \im F$. 
\end{cor}

\begin{proof}
Since $F$ is a full poset morphism we have that $F^{-1}(\up p) \cong \up p \cap \im F$ and $F^{-1}(\down p) \cong \down p \cap \im F$. 
The result follows from \cref{lem-comma-cat-poset}. 
\end{proof}

Final and initial functors, and hence comma categories, are particularly interesting due to their relationship with colimits and limits. 

\begin{thm}\label{thm-final-preserves-colimit}
A functor $F : \cC \to \cD$ is final (resp. initial) if and only if, for all functors $G : \cD \to \cE$, we have $\colim GF \cong \colim G$ (resp. $\lim GF \cong \lim G$). 
\end{thm}

\begin{proof}
This is classical, see for instance \cite[Theorem 1 and Exercise 5 in Chapter IX.3]{maclane} or \cite[Lemma 8.3.4]{Ri14}.
\end{proof}

\section{Generalized Rank and Multiplicity}\label{subsec-generalized-rank}

After introducing the generalized rank, we show in this section a more general version of \cref{thm-rk-decompose} (see below), stated in \cite{kinser2008} for tree quivers and in \cite{chambersletscher2019} for posets, equating the rank of the limit-to-colimit map with the multiplicity of the interval module.

Let $\cP$ be a connected poset and let $M \in \mod \cP$. 
Suppose $(\lambda_p)_{p \in \cP} : \lim M \to M$ and $(\gamma_p)_{p \in \cP} : M \to \colim M$ are the universal (co)cones. 
There is a canonical map $\Psi_{M} : \lim M \to \colim M$ given by $\gamma_p\circ \lambda_p$, for any $p \in \cP$. 
The fact that $\cP$ is connected ensures that $\Psi_M$ does not depend on the choice of $p$. 
The \emph{generalized rank} (\cite{kimmemoli2021}) of $M$ is $\rk M \defeq \rk \Psi_M$, the rank of the limit-to-colimit map.

\begin{thm}\textbf{\emph{(\cite[Lemma 3.1]{chambersletscher2019}).}}\label{thm-rk-decompose}
Every $M \in \mod\cP$ decomposes as $M \cong (\intMod{\cP})^{\rk M}\oplus N$ where $\im\Psi_N = 0$. 
\end{thm}

We now consider a broader setting. 
Let $\cC$ be a small connected category and let $M \in \Mod\cC$. 
We repeat the above construction and define $\rk M$ as the (possibly infinite) dimension of $\im \Psi_M$. 
More precisely, write $(\lambda_c)_{c \in \cC} : \lim M \to M$ and $(\gamma_c)_{c \in \cC}: M \to \colim M$ for the universal (co)cones. 
There is again a canonical map $\Psi_M : \lim M \to \colim M$ given by $\gamma_c\circ \lambda_c$ for any $c \in \cC$ and the fact that $\cC$ is connected ensures that $\Psi_M$ is well-defined. 
Note that since $\Psi_M$ factors through $M(c)$ for any $c\in \cC$, we have that $\rk M \defeq \rk\Psi_M < \infty$ when $M \in \mod\cC$. 

We obtain the following generalization of \cref{thm-rk-decompose}. 

\begin{thm}\label{thm:generalized-rank-general}
Every $M \in \Mod \cC$ decomposes as $M \cong (\intMod{\cC})^{\rk M}\oplus N$ where $\im \Psi_N = 0$. 
\end{thm}

\begin{proof}
Since $\dim_\K \im \Psi_M = \rk M$, we have that $\im \Psi_M \cong \K^{\rk M}$. 
The diagonal functor admits both left and right adjoints, so it preserves products and coproducts and hence $\Delta(\im \Psi_M) \cong \Delta(\K^{\rk M}) \cong (\Delta(\K))^{\rk M} = (\intMod{\cC})^{\rk M}$. 
Therefore, it suffices to show that $M$ decomposes as $M\cong \Delta(\im \Psi_M) \oplus N$ with $\im\Psi_N = 0$. 

Let $c \in \cC$. 
Since every short exact sequence of vector spaces splits, we have the following decompositions in $\Vect_\K$
\begin{equation}
\lim M \cong \im \Psi_M \oplus \ker \Psi_M \text{ and } \colim M \cong \im \Psi_M \oplus \coker \Psi_M.
\label{eq:psi-ses-split}
\end{equation}
We also have an induced isomorphism
\begin{equation*}
\tilde{\gamma}_c\circ\tilde{\lambda}_c : \im \Psi_M \rightarrowtail M(c) \twoheadrightarrow \im\Psi_M
\end{equation*}
where $\tilde{\lambda}_c$ and $\tilde{\gamma}_c$ are the appropriate restrictions of $\lambda_c$ and $\gamma_c$, respectively. 
This implies that 
\begin{equation*}
\im \Psi_M \cong \im \tilde{\lambda}_c \oplus \ker \tilde{\lambda}_c \cong \im \tilde{\lambda}_c    
\end{equation*}
since $\tilde{\lambda}_c$ is necessarily injective. 
Therefore, 
\begin{equation*}
M(c) \cong \im \tilde{\lambda}_c \oplus \coker \tilde{\lambda}_c \cong \im \Psi_M \oplus \coker \tilde{\lambda}_c \cong \Delta(\im \Psi_M)(c) \oplus \coker \tilde{\lambda}_c,
\end{equation*}
where the first isomorphism comes from a split short exact sequences and the third from the definition of $\Delta$.
This pointwise decomposition is preserved by the maps $M(f) : M(c) \to M(d)$. 
Indeed, this follows from the commutativity of the following diagram
\begin{equation}
    \begin{tikzcd}[ampersand replacement=\&]
\& \im \Psi_M \oplus \coker \tilde{\lambda}_c 
  \arrow[dd, "M(f)"] 
  \arrow[rd, "{\gamma_c = \lbracket\begin{smallmatrix}\tilde{\gamma}_c & 0 \\ 0 & g_c\end{smallmatrix}\rbracket}"] 
\& \\
\im \Psi_M \oplus \ker \Psi_M 
  \arrow[rd, "{\lambda_d = \lbracket\begin{smallmatrix}\tilde{\lambda}_d & 0 \\ 0 & \ell_d\end{smallmatrix}\rbracket}"'] 
  \arrow[ru, "{\lambda_c = \lbracket\begin{smallmatrix}\tilde{\lambda}_c & 0 \\ 0 & \ell_c\end{smallmatrix}\rbracket}"] 
\& 
\& \im \Psi_M \oplus \coker \Psi_M \\
\& \im \Psi_M \oplus \coker \tilde{\lambda}_d 
  \arrow[ru, "{\gamma_d = \lbracket\begin{smallmatrix}\tilde{\gamma}_d & 0 \\ 0 & g_d\end{smallmatrix}\rbracket}"'] 
\&
\end{tikzcd}
\label{eq:poset-relations-preserve-rank-decomp}
\end{equation}
where $\ell$ and $g$ are the restrictions to the respective complements of $\im \Psi_M$.
Since cokernels and direct sums are defined pointwise, we have the decomposition $M \cong \Delta(\im \Psi_M) \oplus N$, where $N \defeq \coker \tilde{\lambda}$. 

We now show that $\Psi_N = 0$. 
Write $(\varphi_c)_{c\in \cC} : \lim N \to N$ and $(\rho_c)_{c\in \cC} : N \to \colim N$ for the universal (co)cones. 
These induce (co)cones of $M$, given by 
\begin{equation*}
\lbracket\begin{smallmatrix}0 \\ \varphi_c\end{smallmatrix}\rbracket_{c\in \cC} : \lim N \to \Delta(\im \Psi_M) \oplus N \cong M 
\end{equation*}
and 
\begin{equation*}
\lbracket\begin{smallmatrix}0 & \rho_c\end{smallmatrix}\rbracket_{c\in \cC} : \Delta(\im \Psi_M) \oplus N \cong M \to \colim N.
\end{equation*} 
By the universal property of the limit and the colimit, there are maps $\nu : \lim N \to \lim M$ and $\mu : \colim M \to \colim N$ such that $\lambda_c \circ \nu = \lbracket\begin{smallmatrix}0 \\ \varphi_c\end{smallmatrix}\rbracket$ and $\mu\circ\gamma_c = \lbracket\begin{smallmatrix}0 & \rho_c\end{smallmatrix}\rbracket$ for all $c\in \cC$. 
The first condition implies that $\nu$ factors through $\ker\Psi_M$ and therefore that $\Psi_M \circ \nu = 0$. 
Both conditions together imply $\Psi_N = \mu\circ \Psi_M \circ \nu $ and so $\Psi_N = 0$.
\end{proof}

\begin{cor}
If $M$ is indecomposable, then $\rk M \neq 0$ if and only if $M \cong \intMod{\cC}$. 
\end{cor}
\begin{rmk}
\label{rmk:AGL24}
A detailed proof of \cref{thm-rk-decompose} for $\cC$ a finite connected poset is given in \cite[Lemma~6.17]{asashibagauthierliu2024}. Their argument uses a diagram similar to \cref{eq:poset-relations-preserve-rank-decomp} to show that the split exact sequences in \cref{eq:psi-ses-split} lift to short exact sequences of functors.
In \cite{asashibagauthierliu2024} and in \cite{chambersletscher2019}, the identity $\rk \Psi_M = \rk \Psi_{\im \Psi_M}$ is used to deduce that $\rk \Psi_N = 0$. However, this implication only holds in $\mod \cC$.
\end{rmk}

The following was obtained independently in \cite[Proposition 8.2]{deylesnick2026}. 

\begin{prp}\label{prp-constructionWorks}
The embedding $F : \cM{\rk} \hookrightarrow \cP$ of \cref{construction-M-rk} is such that $\rk M = \rk MF$ for all $M \in \Mod\cP$. 
\end{prp}

\begin{proof}
Let $G^{init}:\sinf^{init} \hookrightarrow \cM{\rk}$ and $G^{fin}:\sinf^{fin} \hookrightarrow \cM{\rk}$ denote the fully faithful embeddings. 
From the construction of $\cM{\rk}$, there are $s_i \in \sinf^{init}$ and $s_f \in \sinf^{fin}$ such that $s_f \geq s_i$ in $\cM{\rk}$. 
We identify $s_i$ and $s_f$ with their images in $\cM{\rk}$ and $\cP$, so $MF G^{fin} (s_f) = MF (s_f)= M (s_f)$ and similarly for $s_i$.
The universal (co)cones of $M$, $MF$, $MF G^{init}$, and $MF G^{fin}$ give rise to respective maps 
\begin{align*}
    \lambda_{s_i} : \lim M \to M(s_i) & \text{ and } \gamma_{s_f} : M(s_f) \to \colim M ,\\
    \lambda_{s_i}' : \lim MF \to M(s_i) & \text{ and } \gamma_{s_f}' : M(s_f) \to \colim MF,\\
    \lambda_{s_i}^{init} : \lim MF G^{init} \to M(s_i) & \text{ and } \gamma_{s_f}^{fin} : M(s_f) \to \colim MF G^{fin}.
\end{align*}

By the universal property of limits, there exist unique morphisms
\begin{equation*}
\lim M \overset{h_1}{\longrightarrow} \lim MF \overset{h_2}{\longrightarrow} \lim MF G^{init}
\end{equation*}
such that $\lambda_{s_i} = \lambda_{s_i} ' \circ h_2 = (\lambda_{s_i}^{init} \circ h_1) \circ h_2$. 
By the universal property of colimits, there exist unique morphisms
\begin{equation*}
 \colim MF G^{fin} \overset{g_1}{\longrightarrow} \colim MF \overset{g_2}{\longrightarrow} \colim M
\end{equation*}
such that $ \gamma_{s_f} = g_1 \circ\gamma_{s_f}' = g_1 \circ (g_2\circ \gamma_{s_f}^{fin}) $. 
Commutativity of these (co)cones ensures that
\begin{equation*}
    \rk \Psi_M 
= \rk(\gamma_{s_f}\circ \lambda_{s_f})
= \rk(\gamma_{s_f} \circ M(s_f \geq s_i)\circ \lambda_{s_i})
\end{equation*}
and identifying $MF(s_f \geq s_i)=M(s_f \geq s_i)$, we get
\begin{equation*}
\rk \Psi_{MF} = \rk(\gamma_{s_f}' \circ MF(s_f \geq s_i)\circ \lambda_{s_i}') = \rk(\gamma_{s_f}' \circ M(s_f \geq s_i)\circ \lambda_{s_i}').
\end{equation*}

We have the following chain of inequalities by linear algebra:
\begin{align}\label{eq:rank-inEquality}
\begin{split}
\rk(\gamma_{s_f}^{fin} \circ M&(s_f \geq s_i)\circ \lambda_{s_i}^{init})\\ &\geq \rk(g_1\circ \gamma_{s_f}^{fin} \circ M(s_f \geq s_i)\circ \lambda_{s_i}^{init}\circ h_2)\\
&= \rk MF\\
& \geq \rk(g_2\circ g_1\circ \gamma_{s_f}^{fin} \circ M(s_f \geq s_i)\circ \lambda_{s_i}^{init}\circ h_2\circ h_1)\\
&= \rk M.
\end{split}
\end{align}
From \cref{thm-colimit-preserving-restriction}, the unique morphisms $h : \lim M \to \lim MF G^{init}$ and $g : \colim MF G^{fin} \to \colim M$ are both isomorphisms. 
By the universal property of (co)limits, $h$ and $g$ are such that $\lambda_{s_i}^{init}\circ h = \lambda_{s_i}$ and $g\circ \gamma_{s_f}^{fin} = \gamma_{s_f}$. 
Hence,
\begin{align}\label{eq:rankEquality}
\begin{split}
\rk M &=\rk(\gamma_{s_f} \circ M(s_f \geq s_i)\circ \lambda_{s_i}) 
\\&= \rk(g\circ \gamma_{s_f}^{fin} \circ M(s_f \geq s_i)\circ \lambda_{s_i}^{init}\circ h) 
\\& = \rk(\gamma_{s_f}^{fin} \circ M(s_f \geq s_i)\circ \lambda_{s_i}^{init}).
\end{split}
\end{align}
Now, \cref{prp-constructionWorks} follows from \cref{eq:rank-inEquality} and \cref{eq:rankEquality}.
\end{proof}

Finally, we prove \cref{prp-multiplicity-preserving-restriction}. 
\begin{prp}\textbf{\emph{(Multiplicity Preserving Restriction).}} 
Let $\cJ$ and $\cC$ be small connected categories. 
If $F : \cJ \to \cC$ is final and initial, then $\mult{\intMod{\cC}}{M} = \mult{\intMod{\cJ}}{MF}$ for all functors $M : \cC \to \Vect_\K$.
\end{prp}

\begin{proof}
By \cref{thm:generalized-rank-general}, $\mult{\intMod{\cC}}{M} = \rk M$ and $\mult{\intMod{\cJ}}{MF} = \rk MF$. 
If $F$ is final and initial, we get from \cref{thm-final-preserves-colimit} 
that limits and colimits are preserved and therefore, the equality $\rk M = \rk MF$ 
holds.
\end{proof}

\section{Proof of \cref{thm-colimit-preserving-restriction} ((Co)limit Preserving Restriction)}\label{sec-proofThmA}

We show \cref{thm-colimit-preserving-restriction} in this section. 
The forward implication of \cref{thm-colimit-preserving-restriction} follows from \cref{thm-final-preserves-colimit}. 
The reverse implication is obtained by combining the following two propositions. 

\begin{prp}\label{prp-theoremA-colim}
Let $F:\cJ \to \cC$ be a functor between small categories and suppose that $\cC$ has finite hom-sets. 
If $\colim MF \cong \colim M$ for all $M \in \mod\cC$, then $F$ is final.
\end{prp}

\begin{proof}
Suppose $F$ is not final.
We show that $\colim MF \not\cong \colim M$ for some $M \in \mod\cC$.
By definition, there is a $c \in \cC$ such that $c/F$ is not connected. 
Consider the functor $\Hom_\cC(c,-) : \cC \to \setcat$ and define $M \defeq \cL \circ \Hom_\cC(c,-)$, where $\cL : \setcat \to \Vect_\K$ is the free functor. 
Since $\cC$ has finite hom-sets, $\Hom_\cC(c,c')$ has finitely many morphisms which implies that $M \in \mod\cC$. 
Note that $\Vect_\K$ has all small limits and colimits. By \cref{equation-colim} in \cref{ex-computation-colim}, we have that $\colim M \cong \K$. 

In case the comma category $c/F$ has several connected components, then $MF$ is the direct sum over all the (non-zero) restrictions of $MF$ to each component, thus $\dim \colim M = 1 \neq 2 \le \dim \colim MF$. 

On the other hand, if the comma category $c/F$ is empty, then $MF$ is the zero module, thus $0 = \colim MF \ncong \colim M$. 
\end{proof}

We now show the dual of the previous proposition.

\begin{prp}\label{prp-theoremA-lim}
Let $F:\cJ \to \cC$ be a functor between small categories and suppose that $\cC$ has finite hom-sets.
If $\lim MF \cong \lim M$ for all $M \in \mod\cC$, then $F$ is initial.
\end{prp}

\begin{proof}
Suppose $F$ is not initial. 
This implies that $F^{op} : \cJ^{op} \to \cC^{op}$ is not final. 
By \cref{prp-theoremA-colim}, there is a $\cC^{op}$-module $M^* : \cC^{op} \to \vect_\K$ such that $\colim M^*F^{op} \ncong \colim M^*$. 
Since the categories $\vect_\K$ and $\vect_\K^{op}$ are equivalent, we get that the categories $(\mod \cC)^{op} = (\funct{\cC}{\vect_\K})^{op} \cong \funct{\cC^{op}}{\vect_\K^{op}} \cong \funct{\cC^{op}}{\vect_\K}$ are also equivalent, and this implies that there is a functor $M^{op} : \cC^{op} \to \vect_\K^{op}$ such that $\colim M^{op}F^{op} \ncong \colim M^{op}$. 
This is equivalent to $\lim MF \ncong \lim M$, with $M \in \mod \cC$. 
\end{proof}

This concludes the proof of \cref{thm-colimit-preserving-restriction}. 
We can now prove \cref{cor-colimit-perserving-restriction-Mod}. 

\begin{cor}
Let $\cJ$ and $\cC$ be small categories and suppose that $\cC$ has finite hom-sets. 
A functor $F: \cJ \to \cC$ is final (resp. initial) if and only if $\colim MF \cong \colim M$ (resp. $\lim MF \cong \lim M$) for all functors $M : \cC \to \Vect_\K$. 
\end{cor}

\begin{proof}
The forward direction follows from \cref{thm-final-preserves-colimit} and the converse from \cref{thm-colimit-preserving-restriction}, since $\mod \cC \subseteq \Mod \cC$.
\end{proof}

\section{Proof of \cref{thm-minimal-final-subposet} (Minimal Final Subposet)}\label{sec-proofThmB}

Suppose that $\cP$ respects \cref{asn-poset-fin}. 
As in the introduction, consider the full subposets with elements $\cS_0 \defeq \max\cP$ and
\begin{equation*}
\cS_n \defeq \cS_{n-1} \sqcup \max\{p \in \cP : \up p \cap \cS_{n-1} \text{ is not connected}\}
\end{equation*}
for $n \geq 1$. 
Note that the union is disjoint because if $s \in \cS_{n-1}$, then $\up s \cap \cS_{n-1}$ is automatically connected. 
Indeed, in that case, every element $t \in \up s \cap \cS_{n-1}$ is such that $s \leq t$. 
We consider $\sinf = \bigcup_{n=0}^\infty S_n$. 

Before proving \cref{thm-minimal-final-subposet}, we give an example of a poset not satisfying \cref{asn-poset-fin} and show that it doesn't admit a minimal final subposet. 

\begin{ex}\label{ex-no-min-fin-sub}
We first note that \cref{asn-poset-fin} (\ref{asn-poset-fin-downset}) implies $0 < \#\max\cP < \infty$, but is a strictly stronger assumption. 
For instance, take 
\begin{equation*}
\cP \defeq \down\{(0,1), (1,0)\}\setminus\{(1,0)\} \subset \bbR^2. 
\end{equation*}
It has a maximum (so $0 < 1 = \#\max\cP < \infty$), but it's not the downset of a finite poset. 
We show that $\cP$ does not have a minimal final subposet. 

Suppose for the purpose of contradiction that $\cQ \subseteq \cP$ is a minimal final subposet, i.e., it is a final subposet with $\Obj\cQ$ contained in every final subposet. 
Denote by $G: \cQ \hookrightarrow \cP$ its embedding. 
We may assume that $\cQ$ is a full subposet, since adding relations doesn't change the fact that $\cQ$ is final and has a minimal set of objects. 
We know from \cref{cor-comma-cat-poset-order-emb} that for any $p \in \cP$, the poset $\up p \cap \im G$, which we identify with $\up p \cap \cQ$, is connected (hence non-empty), so there must be a $q \in \cQ$ such that $p \leq q$ in $\cP$. 
In particular, for all $0<x<1$ there is $(x', 0) \in \Obj\cQ$ with $p= (x,0) \leq (x',0)$ in $\cP$. 
Hence, there is also $(x'',0) \in \Obj\cQ$ such that $(x',0) < (x'',0)$. 

However, removing $(x',0)$ still yields a final subposet, which is a contradiction. 
Indeed, suppose that $(x',0)$ is in a zigzag between two objects in $\up p \cap \cQ$. 
We want to show that $\up p \cap (\cQ \setminus \{(x',0)\})$ contains a zigzag with the same endpoints. 
There are two cases. 
If the zigzag locally looks like $q > (x',0) < q'$, then $q = (q_1,0)$ and $q' = (q_1',0)$ for some $0< q_1, q_1' < 1$. Then $q$ and $q'$ are comparable and $(x',0)$ was not necessary to begin with. 
If the zigzag locally looks like $q < (x',0) > q'$, then by the discussion above there is $(x',0) < (x'',0) \in \up p \cap (\cQ\setminus \{(x',0)\})$, so one can replace $(x',0)$ by $(x'',0)$ and still obtain a zigzag connecting the two elements. 
Thus, $\up p \cap (\cQ\setminus \{(x',0)\})$ is connected for all $p \in \cP$, which concludes the argument.
\end{ex}

Although the subposets $\cS_n$, and thus $\cS_\infty$, can be defined for any poset, they are very well-behaved under \cref{asn-poset-fin}. 
This is in part justified by the following lemma, which will be used in the proof of \cref{prp-min-final-is-final}.

\begin{lem}\label{lem-Sn-max-exists}
For all $n\geq 0$, denote $\cQ_n \defeq \{p\in \cP : \up p \cap \cS_{n} \emph{ is not connected} \}$. 
Then, either $\cQ_n = \emptyset$ or $\max \cQ_n \neq \emptyset$. 
\end{lem}

\begin{proof}
Suppose that $\cQ_n \neq \emptyset$, i.e., $\up p \cap \cS_{n}$ is not connected for some $p \in \cP$. 
Since $\max \cP \subseteq \cS_{n}$, the set $\up p \cap \cS_{n}$ is non-empty. Therefore there exist disconnected elements $s,t \in \up p \cap \cS_{n}$. 
From \cref{asn-poset-fin} (\ref{asn-poset-fin-downset}) we know that there are $m_s,m_t \in \max\cP$ such that $s \leq m_s$ and $t \leq m_t$. 
Therefore, $m_s$ and $m_t$ are disconnected in $\up p \cap \cS_{n}$. 
\cref{asn-poset-fin} (\ref{asn-poset-fin-maximality}) applied to $\cA = \up p \cap \down m_s $ and $\cB = \up p \cap \down m_t $ yields that there exists some $q \in \max (\up p \cap \down m_s \cap \down m_t)$. 
Then $m_s$ and $m_t$ are disconnected in $\up q \cap \cS_{n}$ because they were already disconnected in $\up p \cap \cS_{n}$. 
In particular, $q \in \cQ_n$. 
We will show that $\max \cQ_n$ contains the maximum over such $q$'s. 

By \cref{asn-poset-fin} (\ref{asn-poset-fin-downset}), there are finitely many pairs $m,n \in \max \cP$ and hence finitely many sets of the form
\begin{equation*}
    \mathcal X_{m,n}\defeq\{q \in \max(\up p \cap \down m \cap \down n) : \up q \cap \cS_{n} \text{ is not connected}\}.
\end{equation*}
Therefore we can choose some $x \in \max~ \bigcup_{m,n \in \max\cP} \mathcal X_{m,n} $. 
To prove the lemma, we show that $x \in \max \cQ_n = \max\{\,q \in \cP : \up q \cap \cS_{n} \text{ is not connected}\,\}$. 
Indeed, take any $x' \geq x$ such that $\up x' \cap \cS_{n}$ is disconnected. 
By the above argument, there exist $m', n' \in \max \cP$ that are disconnected in $\up x' \cap \cS_{n}$.
Since $p \leq x'$, we have $\up x' \subseteq \up p$ which implies that $\up x' \cap \down m' \cap \down n' \subseteq \up p \cap \down m' \cap \down n'$. 
Therefore, we have $x' \leq \max\{\up x' \cap \down m' \cap \down n'\} \leq \max\{\up p \cap \down m' \cap \down n'\} \leq x$, so $x=x'$. 
Note that here the inequalities are element-wise.
\end{proof}

We naturally split the proof of \cref{thm-minimal-final-subposet} into three propositions. 
We start with the following. 

\begin{prp}\label{prp:lEqualsMax}
There is an $\ell \leq \#\max\cP - 1$ such that $\cS_\ell = \cS_{\ell+1} = \cdots = \cS_\infty$. 
Furthermore, this bound is sharp. 
\end{prp}

\begin{proof}
To see that the bound in \cref{prp:lEqualsMax} is sharp, we give the following poset, which has $\cS_0 \subsetneq \cS_1 \subsetneq \cdots \subsetneq \cS_{\#\max\cP-1}$.
\begin{equation*}\label{ex:lEqualsMax}
\cP = 
\begin{tikzcd}[sep=0.9em]
& & & & p_{\#\max\cP}\\
& & p_3 \\
& p_2\\
p_1 & \arrow[l] \arrow[u] s_1 & \arrow[l] \arrow[uu] s_2 & \cdots \arrow[l] & s_{\#\max\cP-1}\arrow[l] \arrow[uuu]
\end{tikzcd}
\end{equation*}

We now demonstrate that this constitutes the worst case.
If $\max\cP$ has a unique element, then the statement is obvious. 
Suppose now that $\#\max\cP \geq 2$ and let $n \geq 1$. We write $\cS_{-1} \defeq \emptyset$. 
Suppose that $\cS_n\setminus\cS_{n-1} \neq \emptyset$ and take $s_n \in \cS_n\setminus\cS_{n-1}$. 
We claim that there exist $s_{n-1}, t_{n-1}$ disconnected in $\up s_n \cap \cS_{n-1}$ such that $s_{n-1} \in \cS_{n-1}\setminus \cS_{n-2}$. 

Suppose by contradiction this is not the case, i.e., all $s,t$ disconnected in $\up s_n \cap \cS_{n-1}$ are in $\cS_{n-2}$. 
This implies that these $s,t$ are disconnected in $\up s_n \cap \cS_{n-2}$. 
However, $s_n \notin \max\{p \in \cP : \up p \cap \cS_{n-2} \text{ is not connected}\} = \cS_{n-1}\setminus \cS_{n-2}$ which implies there is a $q \in \cS_{n-1}\setminus \cS_{n-2}$ such that $s_n < q$ which connects some $s,t$. 
But then, $s$ and $t$ are connected in $\up s_n \cap \cS_{n-1}$ via $s \geq q \leq t$, which is a contradiction. 

By repeating the process described above, we obtain a strictly increasing sequence $s_n < s_{n-1} < \cdots < s_1$ together with associated elements $t_{n-1}, t_{n-2}, \ldots, t_1$ as above. 
By \cref{asn-poset-fin} (\ref{asn-poset-fin-downset}), we know that each $t_i$ is such that $t_i \leq p_i$ for at some  $p_i \in \max\cP$. 
Furthermore, we show that $p_i \neq p_j$ if $i \neq j$: 

Suppose to the contrary that $p_i = p_j$ for $i > j$. 
But then the subposet 
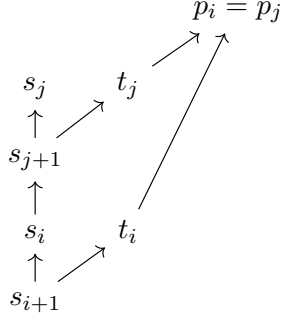
\begin{equation*}
\begin{tikzcd}[sep=0.9em]
& & p_i = p_j\\
s_j & t_j \arrow[ur] & &\\
s_{j+1} \arrow[u] \arrow[ur] & & \\
s_i \arrow[u] & t_i \arrow[uuur] & \\
s_{i+1} \arrow[u] \arrow[ur] & & 
\end{tikzcd}
\end{equation*}
yields a contradiction because  $s_i$ and $t_i$ are connected in $\up s_{i+1}\cap \cS_{i}$. 

By construction, $s_1$ is such that $p_0 > s_1 < p_0'$ for $p_0, p_0' \in \max\cP = \cS_0$ where, by the same argument as above, $p_0, p_0' \neq p_i$ for all $i = 1, \ldots, n-2, n-1$. 
By the pigeonhole principle, $n \leq \#\max\cP - 1$ which concludes the proof of the Proposition. 
\end{proof}

\begin{prp}\label{prp-min-final-is-final}
The embedding $F : \sinf \hookrightarrow \cP$ is final. 
\end{prp}

\begin{proof}
For ease of notation, we identify $\im F$ with $\sinf$. 
Since $\sinf$ is full, $F$ is full. 
Hence, by \cref{cor-comma-cat-poset-order-emb}, we need to show that for all $p \in \cP$, the poset $\up p \cap \sinf$ is connected. 
Suppose by contradiction that $\up p \cap \sinf$ is not connected for some $p \in \cP$. 
By \cref{prp:lEqualsMax}, there is an $\ell \leq \#\max\cP-1$ for which $\up p \cap \cS_\ell = \up p \cap \sinf$ and it follows from \cref{lem-Sn-max-exists} that there exists $x \in \max \{q \in \cP : \up q \cap \cS_\ell \text{ is not connected}\}$. 
This implies that $x \in \cS_{\ell+1}$ and again by \cref{prp:lEqualsMax}, that $x \in \sinf$. 
Therefore, $\up x \cap \sinf$ is connected, since every $q \in \up x \cap \sinf$ is such that $x \leq q$. 
However, we have that $\up x \cap \cS_\ell = \up x \cap \sinf$ is disconnected, a contradiction. 
\end{proof}

\begin{prp}
If $G : \cQ \hookrightarrow \cP$ is a final poset embedding, then $\Obj\sinf \subseteq \Obj\cQ$. 
\end{prp}

\begin{proof}
If $\cQ$ is not a full subposet, we can add relations to obtain a full and final subposet $\cQ'$ with $\Obj \cQ = \Obj \cQ'$. Therefore, we assume without loss of generality that $\cQ$ is full. By \cref{cor-comma-cat-poset-order-emb}, we identify $\cQ$ with $\im G$. 

By hypothesis, $\up p \cap \cQ$ is connected for every $p \in \cP$. 
We will show by induction that $\Obj\sinf \subseteq \Obj\cQ$. 
First, for any $p\in \max \cP$ we have $\up p \cap \cQ =\{p\} \cap \cQ$ is nonempty, so $\cS_0 = \max\cP \subseteq \cQ$. 
Suppose that $\Obj\cS_n \subseteq \Obj\cQ$ for some $n \geq 0$. 
Suppose by contradiction that $\Obj\cS_{n+1} \nsubseteq \Obj\cQ$ and take $x \in \cS_{n+1} \setminus \cQ$. 
Because $\Obj\cS_n \subseteq \Obj\cQ$, we have that $\Obj(\cS_{n+1} \setminus \cQ) \subseteq \Obj(\cS_{n+1}\setminus \cS_n)$ and therefore, $x \in \max\{p \in \cP : \up p\cap \cS_n \text{ is not connected}\}$. 
This implies that there are $s, t \in \up x \cap \cS_n$ disconnected. 
Since $\up x \cap \cQ$ is connected, there is a zigzag $s \geq q_1 \leq q_2 \geq \cdots \leq q_{k-1} \geq q_k \leq t$ in $\up x \cap \cQ$. 
Because $x \notin \cQ$, we have that $x < q_j$ for all $1 \leq j \leq k$. 
Since $x \in \max\{p \in \cP : \up p\cap \cS_n \text{ is not connected}\}$ and $x < q_j$, we know that $\up q_j \cap \cS_n$ is connected for all $1 \leq j \leq k$. 

We claim that $s$ and $t$ are connected in $\up x \cap \cS_n$, which would be a contradiction. 
Since $q_j \leq q_{j+1}$ or $q_j \geq q_{j+1}$, we have that $\up q_j \supseteq \up q_{j+1}$ or $\up q_j \subseteq \up q_{j+1}$ for all $j$ and therefore, $(\up q_j \cap \cS_n) \cup (\up q_{j+1} \cap \cS_n)$ is connected for all $j$. 
This implies that $\cup_{j=1}^k(\up q_j \cap \cS_n) = (\cup_{j=1}^k \up q_j) \cap \cS_n$ is connected. 
However, we have that $s, t \in (\cup_{j=1}^k \up q_j) \cap \cS_n \subseteq \up x \cap \cS_n$ so $s$ and $t$ are connected in $\up x \cap \cS_n$. 

This implies that $\up x \cap \cS_n$ is connected, which contradicts our hypothesis that $\Obj\cS_{n+1} \nsubseteq \Obj\cQ$. 
Hence, $\Obj\cS_{n+1} \subseteq \Obj\cQ$ and consequently, by \cref{prp:lEqualsMax}, $\Obj\cS_{\#\max\cP-1} = \Obj\sinf \subseteq \Obj\cQ$. 
\end{proof}

Before proving \cref{prp-converseB}, we give a counter-example of this result when the poset morphism is not full.

\begin{ex}\label{ex-counter-ex-converse-B}
Consider
\begin{equation*}
\cS = 
\begin{tikzcd}[sep=1.3em]
s_4 & s_3\\
s_1 \arrow[u] \arrow[ur] & s_2\arrow[u]
\end{tikzcd}
\!\qquad {\rm and } \qquad \;
\cP = \!\!
\begin{tikzcd}[column sep=0.4em, row sep=1.3em]
&p_4\\
&p_3 \arrow[u]\\
p_1\arrow[ur] & & p_2 \arrow[ul]
\end{tikzcd}
\end{equation*}
with $F : \cS \to \cP$ given on objects by $F(s_i) = p_i$ for $i = 1,2,3,4$. This functor is not final, since $p_3/F$ is not connected, but $\mult{\intMod{\cP}}{M} = \mult{\intMod{\cS}}{MF}$ for all $M \in \mod\cP$. 
This is a straightforward verification using the twelve distinct isomorphism classes of indecomposable representations of $\cP$. 
This is the phenomenon seen in \cref{ex-explicitcomputation}, where $\cS = \cS_c$ and $\cP=\cS_b$. 
\end{ex}

\begin{rmk}
    The previous example illustrates the possibility that the poset $\cP$ can have more indecomposables than $\cS$, even if it is a refinement of the poset $\cS$---a phenomenon that does not occur when one considers  $\cP$-spaces, that is, vector spaces with a family of subspaces indexed and ordered over $\cP$.
\end{rmk}

We now prove \cref{prp-converseB}.

\begin{prp}
Suppose that $F:\cS \to \cP$ is a full poset morphism between connected posets and that $\cP$ satisfies \cref{asn-poset-fin} (resp. \cref{asn-poset-fin-dual}). If $\mult{\intMod{\cP}}{M} = \mult{\intMod{\cS}}{MF}$ for all $M \in \mod\cP$, then $F$ is final (resp. initial). 
\end{prp}

We only do the proof in the case $\cP$ satisfies \cref{asn-poset-fin}. 
We suppose that $F: \cS \to \cP$ is not final and for all $x \in \cP$, we identify, by \cref{cor-comma-cat-poset-order-emb}, the comma category $x/F$ with $\up x \cap \im F$. 
We split the proof into two cases. 

\begin{prp}\label{prp-empty-case}
    If $\up p \cap \im F$ is empty for some $p\in \cP$, then there exists a persistence module $M : \cP \to \vect_\K$ such that $\mult{\intMod{\cP}}{M} \neq \mult{\intMod{\cS}}{MF}$.
\end{prp}

\begin{proof}
    We take $M$ to be $\intMod{\cP \setminus \up p}$.  
    Since the complement $\cP \setminus \up p$ is a downset, it is convex, and therefore $M$ is a $\cP$-module. 
    By hypothesis, we have $F(s) \in \cP \setminus \up p$, for all $s \in \cS$. Thus, $MF$ is the interval module $\intMod{\cS}$. 
    However, $\mult{\intMod{\cP}}{M} = 0$ because $M(p) = 0$. 
\end{proof}

\begin{prp}\label{prp-nonempty-case}
    Suppose that $\up p \cap \im F$ is nonempty for all $p \in \cP$. 
    Suppose however that there exists $x \in \cP$ such that $\up x \cap \im F$ is not connected. 
    Then, there exists a persistence module $M : \cP \to \vect_\K$ such that $\mult{\intMod{\cP}}{M} \neq \mult{\intMod{\cS}}{MF}$.
\end{prp}

The construction of the module in \cref{prp-nonempty-case} requires three technical lemmas.

\begin{lem}\label{lem:max-in-image}
     For any $z\in \max \cP$, if $z/ F$ is nonempty, then $z\in \im F$.
\end{lem}

\begin{proof}
    We have $\emptyset \neq \up z \cap \im F = \{z\}\cap \im F$, so $z\in \im F$.
\end{proof}

\begin{lem}\label{lem-wlog}
Under the assumptions of \cref{prp-nonempty-case},
there exist $x$, $z$, and $z'$ in $\cP$ such that
\begin{enumerate}
    \item $x \in \max(\down z \cap \down z')$.
    \item\label{lem:maximalElt} $z,z' \in \max\cP$ and are contained in distinct connected components of $\up x \cap \im F$.
    \item\label{item-lem-wlog} For each $x < w < z$ and $x < w' < z'$, we have that $\up w \cap \im F$ and $\up w' \cap \im F$ are connected. 
\end{enumerate}
\end{lem}

\begin{proof}
    From the hypotheses, $\up x_1 \cap \im F$ is disconnected, for some $x_1$. By \cref{lem:max-in-image}, there exists $z_1, z_1' \in \max \cP \subseteq \im F$ contained in distinct connected components of $\up x \cap \im F$.
    Let $i\in \mathbb N$, and suppose there is some $x_i < x_{i+1} < z_i$ such that $\up x_{i+1} \cap \im F$ is disconnected. By \cref{asn-poset-fin} (\ref{asn-poset-fin-downset}), there are $z_{i+1}, z_{i+1}' \in \max \cP$ contained in distinct connected components of $\up x_{i+1} \cap \im F$. 
    
    Suppose for the purpose of contradiction that the sequence $(x_1,z_1,z_1')$, $(x_2,z_2,z_2')$, $\ldots$ with $x_1 < x_2< \cdots$ is infinite. By \cref{asn-poset-fin} (\ref{asn-poset-fin-downset}), $\max \cP$ is finite, so by the pigeonhole principle there are maxima $z_{j}, z_{j}' \in \max \cP$ such that the antichain $\max (\down z_{j}\cap \down z_j')$ contains infinitely many $x_i$. However, this contradicts the fact that the $x_i$ are strictly increasing. Therefore, the sequence terminates at some $(x_i,z_i,z_i')$. Then $x= x_i$, $z=z_i$, and $z' = z_i'$ satisfy all three conditions of the lemma.
\end{proof}

\begin{lem}\label{lem:connectedComp}
Suppose $x$, $z$, and $z'$ satisfy \cref{lem-wlog}.
Then $z$ and $z'$ are in two different connected components of $\up x \setminus \{x\}$ which we denote by $\Gamma$ and $\Gamma'$, respectively. 
\end{lem}

\begin{proof}
Suppose on the contrary that there is a zigzag $z \geq w_1 \leq w_2 \geq \cdots \leq w_{\ell-1} \geq w_\ell \leq z'$ in $\up x \setminus \{x\}$. 
By \cref{lem-wlog}, we know that $\up w_j\cap \im F$ is connected for all $1 \leq j \leq \ell$. 
Since $w_j \leq w_{j+1}$ or $w_j \geq w_{j+1}$, we have that $\up w_j \supseteq \up w_{j+1}$ or $\up w_j \subseteq \up w_{j+1}$ for all $j$ and therefore, $(\up w_j \cap \im F) \cup (\up w_{j+1} \cap \im F)$ is connected for all $j$. 
This implies that $\cup_{j=1}^\ell(\up w_j \cap \im F) = (\cup_{j=1}^\ell \up w_j) \cap \im F$ is connected. 
Because $x < w_j$ for all $j$, we have that $\cup_{j=1}^\ell \up w_j \subseteq \up x$ and therefore $(\cup_{j=1}^\ell \up w_j) \cap \im F \subseteq \up x \cap \im F$. 
However, we have that $z, z' \in (\cup_{j=1}^\ell \up w_j) \cap \im F$, which contradicts \cref{lem-wlog} (\ref{lem:maximalElt}). 
Hence, $z$ and $z'$ are not connected in $\up x \setminus \{x\}$. 
\end{proof}

\begin{proof}[Proof of \cref{prp-nonempty-case}]
We now construct a functor $M : \cP \to \vect_\K$ as follows. Suppose $x$, $z$, and $z'$ satisfy \cref{lem-wlog}. Let $\Gamma$ be one of the connected components of \cref{lem:connectedComp}.
Let $x' \leq x'' \in \cP$. 
We put $M(x) = \K^2$ and $M(x') = \K$ for all $x \neq x'$. 
For the maps, we define 
\begin{equation*}
M(x' \leq x'') \defeq 
\begin{cases}
\left[\begin{smallmatrix}1\\1\end{smallmatrix}\right] & \text{if }x' \neq x \text{ and }x'' = x,\\
\left[\begin{smallmatrix}1 & 0\end{smallmatrix}\right] & \text{if }x' = x \text{ and }x'' \in \Gamma,\\
\left[\begin{smallmatrix}0 & 1\end{smallmatrix}\right] & \text{if }x' = x \text{ and }x'' \notin \Gamma,\\
\id & \text{otherwise}.
\end{cases}
\end{equation*}

The construction is summarized in the diagram below. The top row is the connected component $\Gamma$ and the bottom row is $\left(\up x  \setminus \{x\}\right)\setminus \Gamma$.
\begin{equation*}
\begin{tikzcd}[ampersand replacement=\&] 
\& \& \K \arrow[r, "\id"] \& \K \arrow[r, "\id"] \& \cdots \arrow[r, "\id"] \& M(z) = \K\\
\cdots \arrow[r, "\id"] \& \K \arrow[r, "{\left[\begin{smallmatrix}1\\1\end{smallmatrix}\right]}"] \arrow[ur, "\id", bend left=20] \arrow[dr, "\id", swap, bend right=30] \& M(x) = \K^2 \arrow[ur, "{\left[\begin{smallmatrix}1 & 0\end{smallmatrix}\right]}", swap, bend right=20]
\arrow[dr, "{\left[\begin{smallmatrix}0 & 1\end{smallmatrix}\right]}", bend left=30]
\arrow[u, "{\left[\begin{smallmatrix}1 & 0\end{smallmatrix}\right]}", swap] 
\arrow[d, "{\left[\begin{smallmatrix}0 & 1\end{smallmatrix}\right]}"] \&
\\
\& \& \K \arrow[r, "\id"] \& \K \arrow[r, "\id"] \& \cdots \arrow[r, "\id"] \& M(z') = \K
\end{tikzcd}
\end{equation*}
Since there is no zigzag path in $\up x  \setminus \{x\}$ from $\Gamma$ to $\left(\up x  \setminus \{x\}\right)\setminus \Gamma$, it follows that this diagram commutes, which implies that $M$ is a functor. 

Since $MF$ is the interval module $\intMod{\cS}$ we have that $\mult{\intMod{\cS}}{MF} = 1$. 
On the other hand, \cref{lem:connectedComp} implies that the universal cocone is given by the maps
\begin{equation*}
\begin{tikzcd}
M(z) \arrow[dr, "\gamma_{z}", swap] & M(x) \arrow[l, "\lbracket\begin{smallmatrix}1 \amspand 0\end{smallmatrix}\rbracket", swap] \arrow[r, "\lbracket\begin{smallmatrix}0 \amspand 1\end{smallmatrix}\rbracket"] \arrow[d, "\gamma_x"] & M(z') \arrow[dl, "\gamma_{z'}"]\\
& \colim M
\end{tikzcd}
\end{equation*}
such that $\gamma_{z}\left[\begin{smallmatrix}1 & 0\end{smallmatrix}\right] = \gamma_x = \gamma_{z'}\left[\begin{smallmatrix}0 & 1\end{smallmatrix}\right] $ which implies that $\gamma_{z} = \gamma_{z'} = 0$. 
It follows that $\colim M = 0$ and therefore, by \cref{thm:generalized-rank-general}, $\rk M = 0 = \mult{\intMod{\cP}}{M} \neq \mult{\intMod{\cS}}{MF}$. 
\end{proof}

\bibliographystyle{alpha}
\bibliography{main}

\end{document}